\documentclass[12pt,letterpaper]{amsart}
\usepackage[utf8]{inputenc}
\usepackage[margin=1in,footskip=0.25in]{geometry}
\usepackage{amsmath, amssymb, amsthm}
\usepackage{graphicx}
\usepackage{xcolor}
\usepackage{bm}
\usepackage{hyperref}
\usepackage{subfig}
\usepackage{cite}
\usepackage{cleveref}

\newcommand{\scom}{simplicial complex}
\newcommand{\K}{\mathcal{K}}

\newcommand\norm[1]{\lVert#1\rVert}
\newcommand{\low}{\text{low}}
\newcommand{\pdgm}{\text{PD}}
\newcommand{\homdim}{q}
\newcommand{\base}{B}
\newcommand{\bp}{p}
\newcommand{\idx}{\text{idx}}
\newcommand{\tri}{\Delta}

\DeclareMathOperator*{\argmin}{arg\,min}

\newtheorem{theorem}{Theorem}[section]

\theoremstyle{definition}
\newtheorem{definition}[theorem]{Definition}

\newtheorem{lemma}[theorem]{Lemma}
\newtheorem{proposition}[theorem]{Proposition}

\title{Computing Persistence Diagram Bundles}
\author{Abigail Hickok}
\date{\today}

\begin{document}
\begin{abstract}
     Persistence diagram (PD) bundles, a generalization of vineyards, were introduced as a way to study the persistent homology of a set of filtrations parameterized by a topological space $\base$. In this paper, we present an algorithm for computing piecewise-linear PD bundles, a wide class that includes many of the PD bundles that one may encounter in practice. Full details are given for the case in which $\base$ is a triangulated surface, and we outline the generalization to higher dimensions and other cases.
\end{abstract}

\maketitle

\section{Introduction}\label{sec:intro}

In persistent homology, one is given a filtration (e.g., a filtered complex), and one studies how the homology changes as the filtration parameter increases. Here, we consider the case in which one is instead given a \emph{fibered filtration function}, which is a set $\{f_\bp : \K^\bp \to \mathbb{R}\}_{\bp \in \base}$ of filtration functions that is parameterized by some topological space $\base$ (the \emph{base space}), where $\K^\bp$ is a simplicial complex for each $\bp \in \base$. At each $\bp \in \base$, the sublevel sets of $f_\bp$ form a filtered complex. The associated \emph{persistence diagram (PD) bundle} is the space of persistence diagrams $\pdgm(f_\bp)$ as they vary with $\bp \in \base$.

PD bundles arise naturally in many circumstances. The prototypical PD bundle is a \emph{vineyard}, introduced in \cite{vineyards}, which is the case where $\base$ is an interval in $\mathbb{R}$. For example, given a time-varying point cloud, one obtains a vineyard by constructing a filtration (such as the Vietoris--Rips filtration) at each time. Figure \ref{fig:vineyard} shows an illustration of a vineyard, visualized as a continuously-varying ``stack of PDs''. More generally, however, one might have a set $\{X(\bp)\}_{\bp \in \base}$ of point clouds parameterized by an arbitrary parameter space $\base$. One obtains a PD bundle by constructing a filtration for the point cloud at each $\bp \in \base$. For example, biological aggregation models (e.g., the D'Orsogna model \cite{dorsogna}) produce time-varying point clouds whose coordinates also depend on various real-valued system parameters $\mu_1, \ldots, \mu_k$. The parameter space $\base$ in this example is a subset of $\mathbb{R}^{k+1}$. For each $(t, \mu_1, \ldots, \mu_k) \in \base$, there is a point cloud from which one can construct a filtration. Another special case is the persistent homology transform ($\base = \mathbb{S}^d$), which is used in the field of shape analysis \cite{pht}. Other concrete examples of PD bundles are given in \cite{pd_bundle}.

\begin{figure}
	\centering
	\includegraphics[width = .35\textwidth]{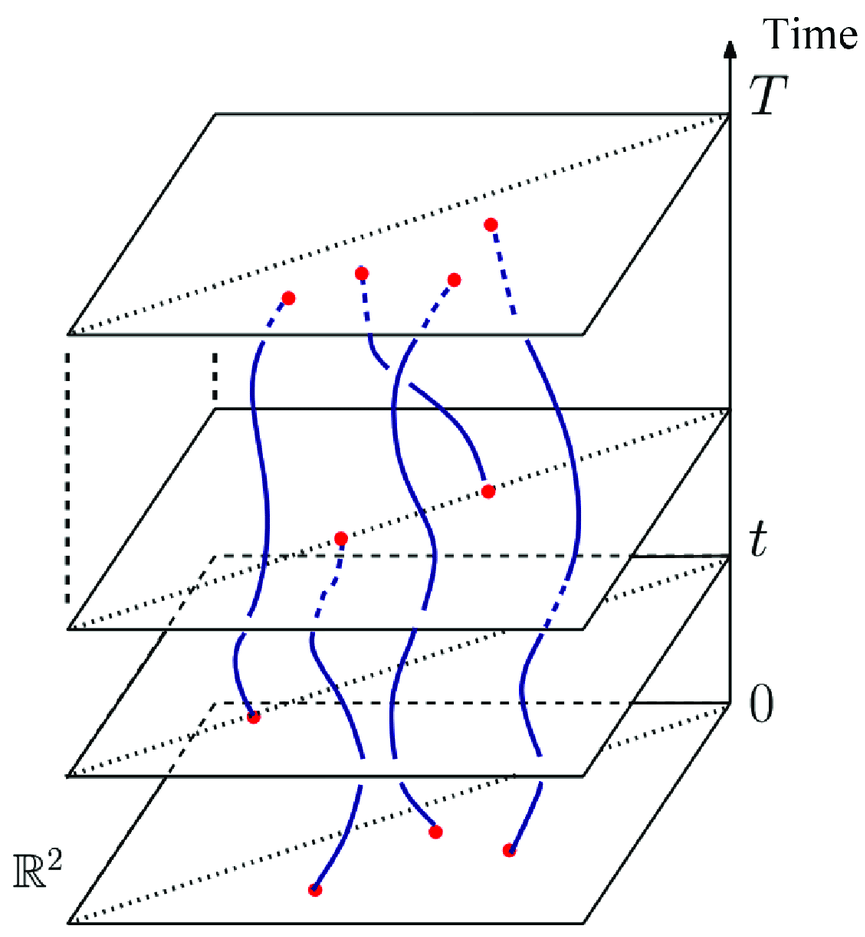}
	\caption{An illustration of a vineyard, which consists of a persistence diagram for each time $t$. (This figure is a slightly modified version of a figure that appeared originally in \cite{vineyard_figure}, which is available under a Creative Commons license.)}
	\label{fig:vineyard}
\end{figure}

\subsection{Contributions}

I generalize Cohen-Steiner et. al's algorithm for computing vineyards \cite{vineyards} to an algorithm for efficiently computing PD bundles. I restrict to the case in which the PD bundle is \emph{piecewise linear}. This means that $\base$ is a simplicial complex, $\K^\bp \equiv \K$ for all $\bp \in \base$, and for every simplex $\sigma \in \K$, the function $f_{\sigma}(\bp) := f_\bp(\sigma)$ is linear in $\bp$ on every simplex of $\base$.
The restriction to piecewise-linear PD bundles allows us to take advantage of methods in computational geometry, such as the Bentley--Ottman planesweep algorithm \cite{compgeo} for finding intersections of lines in a plane and algorithms for solving the point-location problem in a line arrangement \cite{slab}. An analogous piecewise-linear restriction was helpful for computing vineyards in \cite{vineyards}.

The idea of the algorithm is to subdivide the base space $\base$ into polyhedrons and compute a PD ``template'' for each polyhedron. The subdivision is given by \Cref{prop:polyhedrons_constant} (\cite{pd_bundle}). For any $\bp \in \base$, the persistence diagram $\pdgm(f_\bp)$ can then be computed in $\mathcal{O}(N)$ time from the PD template for the polyhedron that contains $\bp$, where $N$ is the number of simplices in $\K$. By contrast, computing PD$(f_\bp)$ from scratch takes $\mathcal{O}(N^3)$ time in the worst case.

The piecewise-linear restriction is reasonable for most applications. For example, suppose that we have a point cloud $X(t, \mu)$ whose coordinates depend on time $t$ and a parameter $\mu \in \mathbb{R}$. If the data set arises from either real-world data collection or through numerical simulation, then we likely only know the coordinates of the point cloud $X(t, \mu)$ at a discrete set $\{t_i\}$ of time steps and a discrete set $\{\mu_j\}$ of system-parameter values. For every $(t_i, \mu_j)$, there is the filtration function $f_{(t_i, \mu_j)}:\K \to \mathbb{R}$ associated with the Vietoris--Rips filtration (or any other filtration) of $X(t_i, \mu_j)$. For the Vietoris--Rips filtration, $\K$ is the simplicial complex that contains a simplex for every subset of points in the point cloud. To obtain a fibered filtration function, we define $\base$ to be a triangulation of $[\min t_i, \max t_i] \times [\min \mu_j, \max \mu_j]$ whose vertices are $\{(t_i, \mu_j)\}_{ij}$. We can extend $\{f_{(t_i, \mu_j)}\}_{ij}$ to a fibered filtration function with base space $\base$ by defining the filtration values of a simplex $\sigma$ via linear interpolation of $\{f_{(t_i, \mu_j)}(\sigma)\}_{ij}$. By construction, the resulting PD bundle is piecewise linear.

Full details are only given for the case in which $\dim(\base)$ is a triangulated surface, but the generalization to higher dimensions is discussed in Section \ref{sec:high}. When the base space $\base$ is a triangulated surface, it is already an improvement over a vineyard because it allows three parameters in total: a filtration parameter as well as two parameters that locally parameterize $\base$.

\subsection{Related Work}
PD bundles were introduced in \cite{pd_bundle} as a generalization of vineyards \cite{vineyards}. The algorithm that I present in this paper for computing PD bundles is a broad generalization of the algorithm in \cite{vineyards}. The algorithm in this paper is also related to the Rivet algorithm for computing fibered barcodes of 2D multiparameter persistence modules \cite{rivet}.

\subsection{Organization}
The paper proceeds as follows. I review the relevant background on persistent homology, vineyards, and PD bundles in Section \ref{sec:background}. I present my algorithm for computing piecewise-linear PD bundles in Section \ref{sec:compute}. Finally, I conclude and discuss possible directions for future research in Section \ref{sec:conclusion}.
\section{Background}\label{sec:background}
We begin by reviewing persistent homology, vineyards, and PD bundles; for more details on persistent homology, see \cite{edel_book, roadmap}, for more on vineyards, see \cite{vineyards}, and for an introduction to PD bundles, see \cite{pd_bundle}.

\subsection{Persistent homology}\label{sec:PH}

Let $\K$ be a simplicial complex. A \emph{filtration function} $f: \K \to \mathbb{R}$ is a real-valued function on $\K$ that is \emph{monotonic}. That is, $f(\tau) \leq f(\sigma)$ if $\tau$ is a face of $\sigma$. Monotonicity guarantees that the $r$-sublevel sets $\K_r := \{\sigma \in \K \mid f(\sigma) \leq r)\}$ are simplicial complexes.

In persistent homology, we study how the homology of $\K_r$ changes as $r$ increases. Let $\{r_i\}_{i=1}^k$ be the image of $f$, ordered such that $r_i < r_{i+1}$. These are the critical values at which $\K_r$ changes; for $r \in [r_i, r_{i+1})$, we have $\K_r = \K_{r_i}$. For every $i \leq j$, the inclusion $\iota^{i, j}: \K_{r_i} \hookrightarrow \K_{r_j}$ induces a map $\iota^{i, j}_*: H_*(\K_{r_i}, \mathbb{F}) \to H_*(\K_{r_j}, \mathbb{F})$ on homology, where $\mathbb{F}$ is a field. For the remainder of this paper, we compute homology over the field $\mathbb{F} = \mathbb{Z}/2\mathbb{Z}$. The \emph{$\homdim$th-persistent homology} (PH) is the pair 
\begin{equation*}
    \Big( \{H_\homdim(\K_{r_i}, \mathbb{F})\}_{1 \leq i \leq k}\,, \{\iota_*^{i, j}\}_{1 \leq i \leq j \leq k} \Big)\,.
\end{equation*}
A homology class is \emph{born} at $r_i$ if it is not in the image of $\iota_*^{i, i-1}$. The homology class \emph{dies} at $r_j > r_i$ if $j$ is the minimum index such that $\iota_*^{i, j}$ maps the homology class to zero. (Such a $j$ may not exist; in that case, the homology class never dies.)
The Fundamental Theorem of Persistent Homology yields compatible choices of bases for the vector spaces $H_\homdim(\K_{r_i}, \mathbb{F})$. The generators in our definition of a persistence diagram, below, are the basis elements in the decomposition given by the Fundamental Theorem of Persistent Homology.

Persistent homology is often visualized as a \emph{persistence diagram} (PD). The $\homdim$th persistence diagram $PD_\homdim(f)$ is a multiset of points in the extended plane $\overline{\mathbb{R}}^2$ that summarizes the $\homdim$th persistent homology. It contains the points on the diagonal with infinite multiplicity (for technical reasons) and one point for every generator. If a generator is born at $b$ and dies at $d$, then the coordinates of the corresponding point in the PD are $(b, d)$, and if the generator is born at $b$ and never dies, then the coordinates of the point are $(b, \infty)$. Note that off-diagonal points in a PD may also appear with multiplicity.

One of the standard methods for computing PH is the algorithm introduced in \cite{ph_alg}, which we review here. Let $f: \K \to \mathbb{R}$ be a filtration function, and let $\sigma_1, \ldots, \sigma_N$ be the simplices of $\K$, indexed such that $i < j$ if $\sigma_i$ is a proper face of $\sigma_j$. The algorithm requires a choice of \emph{compatible simplex indexing} $\idx: \K \to \{1, \ldots, N\}$, where $N$ is the number of simplices in $\K$.
\begin{definition}\label{def:spx_idx_general}
    A \emph{simplex indexing} is a bijection $\idx: \K \to \{1, \ldots, N\}$.
\end{definition}
\begin{definition}\label{def:spx_indexing}
   A \emph{compatible simplex indexing} is a simplex indexing $\idx : \K \to \{1, \ldots, N\}$ such that $\idx(\sigma_i) < \idx(\sigma_j)$ if $f(\sigma_i) < f(\sigma_j)$ or $\sigma_i$ is a proper face of $\sigma_j$.
\end{definition}
 \noindent A compatible simplex indexing exists because monotonicity ensures that $f(\sigma_i) \leq f(\sigma_j)$ if $\sigma_i$ is a face of $\sigma_j$. Because there may not be a unique compatible simplex indexing, we define the \emph{simplex indexing induced by $f$} as follows.
 \begin{definition}\label{def:idx_f}
 The \emph{simplex indexing induced by $f$} is the unique compatible simplex indexing $\idx_f: \K \to \{1, \ldots, N \}$ such that $\idx_f(\sigma_i) < \idx_f(\sigma_j)$ if either $f(\sigma_i) < f(\sigma_j)$ or if $f(\sigma_i) = f(\sigma_j)$ and $i < j$.
\end{definition}
\noindent The function $\idx_f$ is a compatible simplex indexing because the sequence $\sigma_1, \ldots, \sigma_N$ of simplices is ordered such that $i < j$ if $\sigma_i$ is a proper face of $\sigma_j$.
 
Let $D$ be the boundary matrix compatible with the simplex indexing $\idx_f$. That is, let $D$ be the matrix whose $(i, j)$th entry is
\begin{equation*}
    D_{ij} = \begin{cases}
        1\,, & \idx_f^{-1}(i) \text{ is an $(m-1)$-dimensional face of the $m$-dimensional simplex }\idx_f^{-1}(j) \\
        0 \,, & \text{otherwise.}
    \end{cases}
\end{equation*}
We decompose the boundary matrix $D$ into a matrix product $D = RU$ such that $U$ is upper triangular and $R$ is a binary matrix that is ``reduced''. A binary matrix $R$ is \emph{reduced} if $\low_R(j) \neq \low_R(j')$ whenever $j \neq j'$ are the indices of nonzero columns in $R$. The quantity $\low_R(j)$, which is called the \emph{pairing function}, is the row index of the last $1$ in column $j$ if column $j$ is nonzero and undefined if column $j$ is the zero vector. An RU decomposition can be computed in $\mathcal{O}(N^3)$ time \cite{ph_alg, edel_book}.

Cohen-Steiner et al. \cite{vineyards} showed that the pairing function $\low_R(j)$ depends only on the boundary matrix $D$ and not on the particular reduced binary matrix $R$ in the decomposition $D = RU$. A pair $(\idx_f^{-1}(i), \idx_f^{-1}(j))$ of simplices with $i = \low_R(j)$ represents a persistent homology generator. The \emph{birth simplex} $\idx_f^{-1}(i)$ creates the homology class and the \emph{death simplex} $\idx_f^{-1}(j)$ destroys the homology class. The two simplices in a pair have consecutive dimensions; that is, if dim$(\idx^{-1}(i)) = \homdim$, then dim$(\idx^{-1}(j)) = \homdim + 1$. If dim$(\idx_f^{-1}(i)) = \homdim$ and dim$(\idx_f^{-1}(j)) = \homdim + 1$, then a point with coordinates $(f(\idx_f^{-1}(i)), f(\idx_f^{-1}(j)))$ is added to the $\homdim$th persistence diagram. We refer to $f(\idx_f^{-1}(i))$ as its \emph{birth} and to $f(\idx_f^{-1}(j))$ as its \emph{death}. Some simplices are not paired. If $i \neq \low_R(j)$ for all $j$, then the simplex $\idx_f^{-1}(i)$ is a birth simplex for a homology class that never dies. Its birth is $f(\idx_f^{-1}(i))$ and its death is $\infty$. If $\dim(\idx_f^{-1}(i)) = \homdim$, then a point with coordinates $(f(\idx_f^{-1}(i)), \infty)$ is added to the $\homdim$th persistence diagram.

\subsection{Vineyards}\label{sec:vineyards}
Let $\K$ be a \scom. A \emph{$1$-parameter filtration function} on $\K$ is a function $f: \K \times I \to \mathbb{R}$, where $I = [t_0, t_1]$ is an interval in $\mathbb{R}$, such that $f(\cdot, t)$ is a filtration function on $\K$ for all $t \in I$.
For each $t \in I$, the $r$-sublevel sets $\K_r^t = \{\sigma \in \K \mid f(\sigma, t) \leq r\}$ are a filtration of $\K$. The set $\{\{\K_r^t\}_{r \in \mathbb{R}}\}_{t \in I}$ is a set of filtrations parameterized by $t \in I$. For each $t \in I$, one can compute the persistence diagram $PD(f(\cdot, t))$. The associated \emph{vineyard} is the 1-parameter set $\{PD(f(\cdot, t))\}_{t \in I}$ of persistence diagrams. We visualize the vineyard in $\overline{\mathbb{R}}^2 \times I$ as a continuous stack of PDs (see \Cref{fig:vineyard}). The points in the PDs trace out curves with time; these curves are the \emph{vines}.

An algorithm for computing vineyards is given by \cite{vineyards}, and we review it here. Let $f:\K \times I \to \mathbb{R}$ be a $1$-parameter filtration function, and let $\sigma_1, \ldots, \sigma_N$ be the simplices of $\K$, indexed such that $i < j$ if $\sigma_i$ is a proper face of $\sigma_j$. As in \Cref{sec:PH}, we fix a compatible simplex indexing $\idx_f: \K \times I \to \{1, \ldots, N\}$ such that $\idx_f(\sigma_i, t) < \idx_f(\sigma_j, t)$ if we have either $f(\sigma_i, t) < f(\sigma_j, t)$ or we have $f(\sigma_i, t) = f(\sigma_j, t)$ and $i < j$. The simplex indexing can only change at times $t_*$ at which $f(\sigma, t_*) = f(\tau, t_*)$ for some $\sigma, \tau \in \K$. At $t_*$, the indices of $\sigma$ and $\tau$ may change. (If $\sigma$, $\tau$ are the unique pair of simplices whose indices change at $t_*$, then $\sigma$ and $\tau$ are transposed in the simplex indexing and $\sigma$ and $\tau$ have consecutive indices in the indexing. Otherwise, there is a sequence of such transpositions.) Let $D(t)$ denote the boundary matrix compatible with the indexing at time $t$ and let $\low_R(\cdot, t)$ denote the corresponding pairing function. One computes $D(t_0)$ at the initial time $t_0$ and an RU decomposition $D(t_0) = R(t_0)U(t_0)$. The initial pairing function $\low_R(\cdot, t_0)$ is read off from the initial RU decomposition. Then we sweep through the intersections $t_*$ (from left to right) at which the simplex indexing changes. At each $t_*$, we update the simplex indexing, RU decomposition, and pairing function. This updating procedure yields the birth and death simplices for the filtration function $f(\cdot, t_*)$, which one can use to obtain the new persistence diagram.

The procedure above is an efficient way of computing the diagrams PD$(f(\cdot, t))$ for all $t \in [t_0, t_1]$. At worst, updating $R(t_*)$ requires adding one column to another and adding one row to another---similarly for $U(t_*)$. The addition of columns and rows is an $\mathcal{O}(N)$ operation, although in experiments, the authors of \cite{vineyards} found that updating $R(t_*)$ and $U(t_*)$ can be done in approximately constant time if one uses the sparse matrix representations that are given in \cite{vineyards}. If there is a single transposition of simplices at $t_*$, then at most two (birth, death) simplex pairs are updated, and these updates occur in $\mathcal{O}(1)$ time.

A special type of vineyard is a \emph{piecewise-linear vineyard}. If we are only given $f(\sigma, t_i)$ at discrete time steps $t_i$, then for all $i$ we extend $f(\sigma, t)$ to $t \in [t_i, t_{i+1}]$ by linear interpolation. In this case, one can compute the transposition times $t_*$ by using the Bentley--Ottman planesweep algorithm \cite{compgeo}. This is because computing when (if) two simplices $\sigma, \tau$ get transposed in $[t_i, t_{i+1}]$ is equivalent to finding the intersection (if it exists) between the line segments
\begin{align*}
	y = \frac{f(\sigma, t_{i+1}) - f(\sigma, t_i)}{t_{i+1} - t_i}(t - t_i) + f(\sigma, t_i)\,, \\
	y = \frac{f(\tau, t_{i+1}) - f(\tau, t_i)}{t_{i+1} - t_i}(t - t_i) + f(\tau, t_i)\,.
\end{align*}

\subsection{PD bundles}\label{sec:pd_bundle}
PD bundles were introduced in \cite{pd_bundle} as a generalization of vineyards in which a set of filtrations is parameterized by an arbitrary topological space $\base$. A vineyard is the special case in which $B$ is an interval in $\mathbb{R}$.
\begin{definition}\label{def:pdbundle}
A \emph{fibered filtration function} is a set $\{f_\bp: \K^\bp \to \mathbb{R}\}_{\bp \in B}$ of filtration functions parameterized by a topological space $B$ (the \emph{base space}). When $\K^\bp \equiv \K$ for all $\bp \in B$, we define $f(\sigma, \bp) := f_\bp(\sigma)$ for all simplices $\sigma \in \K$ and $\bp \in B$.
\end{definition}

\begin{definition}
Let $\{f_\bp: \K^\bp \to \mathbb{R}\}_{\bp \in B}$ be a fibered filtration function. The topological space $B$ is the \emph{base space}. The space $E := \{(\bp, z) \mid z \in PD_\homdim(f_\bp)\,, \bp \in B\}$ is the $\homdim$th \emph{total space}. We give $E$ the subspace topology inherited from the inclusion $E \hookrightarrow B \times \overline{\mathbb{R}}^2$. The associated \emph{$\homdim$th PD bundle} is the triple $(E, B, \pi)$, where $\pi$ is the projection from $E$ to $B$. 
\end{definition}

In \cite{vineyards}, it was computationally easier to work with a piecewise-linear vineyard, which is a vineyard for a fibered filtration function of the form $f: \K \times [t_0, t_1] \to \mathbb{R}$ such that $f(\sigma, \cdot)$ is piecewise linear for all $\sigma \in \K$. (See the discussion at the end of \Cref{sec:vineyards}.) Below, we define an analog of piecewise-linear vineyards.

\begin{definition}[Piecewise-linear PD bundles]
    Let $\{f_\bp : \K^\bp \to \mathbb{R}\}_{\bp \in B}$ be a fibered filtration function in which $\K^\bp \equiv \K$. As in \Cref{def:pdbundle}, we define $f(\sigma, \bp) := f_\bp(\sigma)$ for all $\sigma \in \K$ and $\bp \in B$. If $B$ is a simplicial complex and $f(\sigma, \cdot)$ is linear on each simplex of $B$ for all simplices $\sigma \in \K$, then $f$ is a \emph{piecewise-linear fibered filtration function.} The resulting PD bundle is a \emph{piecewise-linear PD bundle}.
\end{definition}
\noindent For the remainder of the paper, we only consider piecewise-linear PD bundles.

For example, in \Cref{sec:intro} we considered a point cloud $X(t, \mu)$ whose coordinates depended on time $t \in \mathbb{R}$ and parameter $\mu \in \mathbb{R}$. Given only the coordinates of the point cloud at a discrete set $\{t_i\}$ of time steps and a discrete set $\{\mu_j\}$ of parameter values, we obtained a filtration function $f_{(t_i, \mu_j)}$ for every $(t_i, \mu_j)$. We extended this to a piecewise-linear fibered filtration function on $B = [\min t_i, \max t_i] \times [\min \mu_j, \max \mu_j]$ via linear interpolation of the filtration values for each simplex $\sigma \in \mathcal{K}$.

More generally, suppose that we are given a fibered filtration function of the form $f: \K \times \prod_{i=1}^m \mathcal{I}_i \to \mathbb{R}$, where each $\mathcal{I}_i$ is a finite subset of $\mathbb{R}$, and we wish to extend $f$ to a fibered filtration function whose base is $B := \prod_{i=1}^m [\min \mathcal{I}_i, \max \mathcal{I}_i]$. First, we construct a triangulation $B$ (i.e., an $m$-dimensional \scom) of $\prod_{i=1}^m [\min \mathcal{I}_i, \max \mathcal{I}_i]$ whose set of vertices is $\prod_{i=1}^m \mathcal{I}_i$. (See \cite{cube_triangulation}, for example, for a method to triangulate a cubical complex.) We then extend $f$ to a piecewise-linear fibered filtration function $f: \K \times B \to \mathbb{R}$ by linearly interpolating $f(\sigma, \cdot)$ on each simplex $\Delta \in B$ for all simplices $\sigma \in \mathcal{K}$.

In \cite{pd_bundle}, it was shown that if $f: \K \times B \to \mathbb{R}$ is a piecewise-linear fibered filtration function on an $n$-dimensional simplicial complex $B$, then $B$ can be subdivided into $n$-dimensional polyhedra such that within each polyhedron $P$, there is a ``template'' from which $PD_\homdim(f(\cdot, \bp))$ can be computed for any $\bp \in P$. The template is a list of (birth, death) simplex pairs $(\sigma_b, \sigma_d)$.

The polyhedra are defined as follows. We define
\begin{equation}\label{eq:Idef}
    I(\sigma, \tau) := \{\bp \in B \mid f(\sigma, \bp) = f(\tau, \bp)\} \,.
\end{equation}
For every $n$-simplex $\Delta$ in $B$, the intersection $I(\sigma, \tau) \cap \Delta$ is $\emptyset$, $\Delta$, a vertex of $\Delta$, or the intersection of an $(n-1)$-dimensional hyperplane with $\Delta$. The set
\begin{equation}\label{eq:partition}
    \bigcup_{\Delta \in B} \partial \Delta \cup \Big\{\Big(I(\sigma, \tau) \cap \Delta \Big) \, \Big\vert \, \emptyset \subset \Big(I(\sigma, \tau) \cap \Delta\Big) \subset \Delta \Big\}
\end{equation}
determines the boundaries of a polyhedral decomposition of $\base$, where $\Delta$ denotes an $n$-simplex of $B$ and $\partial \Delta$ denotes the boundary of $\Delta$. Every polyhedron face is either a subset of $\partial \Delta$ or a subset of $I(\sigma, \tau) \cap \Delta$ for some simplex $\Delta$ of $B$.

As in Sections \ref{sec:PH} and \ref{sec:vineyards}, we define the \emph{simplex indexing induced by $f$} as follows. Let $\sigma_1, \ldots, \sigma_N$ be the simplices of $\K$, indexed such that $i < j$ if $\sigma_i$ is a proper face of $\sigma_j$. We define the simplex indexing function $\idx_f: \K \times B \to \{1, \ldots, N\}$ to be the unique function such that $\idx_f(\sigma_i, \bp) < \idx_f(\sigma_j, \bp)$ if we have $f(\sigma_i, \bp) < f(\sigma_j, \bp)$ or we have $f(\sigma_i, \bp) = f(\sigma_j, \bp)$ and $i < j$.

\begin{proposition}[\cite{pd_bundle}]\label{prop:polyhedrons_constant}
If $f: \K \times B \to \mathbb{R}$ is a piecewise-linear fibered filtration function, then the set in \Cref{eq:partition} subdivides $B$ into polyhedra $P$ on which the simplex indexing is constant (i.e., $\idx_f(\sigma, \cdot)\vert_P$ is constant for all $\sigma \in \K$).
\end{proposition}

\noindent \Cref{prop:polyhedrons_constant} implies that within each polyhedron $P$, the set $\{(\sigma_b, \sigma_d)\}$ of (birth, death) simplex pairs for $f(\cdot, \bp)$ is constant with respect to $\bp$.
\section{Computing piecewise-linear PD bundles}\label{sec:compute}

The algorithm to compute piecewise-linear PD bundles has three main parts.
\begin{enumerate}
	\item {\bf Compute the polyhedra:} First, we compute the polyhedra on which the simplex indexing (and thus pairing function) is constant (see \Cref{prop:polyhedrons_constant}). 
	For every pair $P_1$, $P_2$ of adjacent polyhedra, the face $Q$ that they share is a subset of a set of the form $I(\sigma_{i_1}, \sigma_{j_1}) \cap \cdots \cap I(\sigma_{i_m}, \sigma_{j_m})$ for some $m$, where $I(\sigma_{i_1}, \sigma_{j_1})$ is defined by \Cref{eq:Idef}. In the ``generic case,'' defined below at the beginning of \Cref{sec:findpolygons}, we have $m = 1$. We compute and store a reference that is associated with $Q$ to the set $\{(\sigma_{i_1}, \sigma_{j_1}), \ldots, (\sigma_{i_m}, \sigma_{j_m})\}_{k=1}^m$.

\vspace{3mm}

	\item {\bf Compute the pairing function:} Within each polyhedron $P$, the set of (birth, death) simplex pairs for $f(\cdot, \bp)$ is constant with respect to $\bp \in P$. We compute the set of (birth, death) simplex pairs for each $P$. First, we choose an initial point $\bp_* \in \base$ at which we compute the simplex indexing at $\bp_*$, the boundary matrix $D(\bp_*)$, an RU decomposition $D(\bp_*) = R(\bp_*)U(\bp_*)$, and the pairing function at $\bp_*$. (This is the pairing function for the entire polyhedron that contains $\bp_*$.) These computations take $\mathcal{O}(N^3)$ time. We then traverse the polyhedra, starting with the polyhedron that contains $\bp_*$ and visiting each polyhedron at least once. As we move from one polyhedron to the next via a shared face $Q$, we use the set $\{(\sigma_{i_1}, \sigma_{j_1}), \ldots, (\sigma_{i_m}, \sigma_{j_m})\}_{k=1}^m$ computed earlier for $Q$ to update the RU decomposition and pairing function via the update rules that are used when computing vineyards (see \cite{vineyards}). For each polyhedron, we store its pairing function (i.e., the pairs $(\sigma_b, \sigma_d)$ of birth and death simplex pairs and also the unpaired simplices $\sigma_b$, which are birth simplices for homology classes that never die).

 \vspace{3mm}
 
	\item {\bf Query the PD bundle:} To see the $\homdim$th persistence diagram $\pdgm_\homdim(f(\cdot, \bp))$ associated with point $\bp \in \base$, we first locate the polyhedron $P$ that contains $\bp$. For each pair $(\sigma_b, \sigma_d)$ of simplices in the pairing function for $P$, the diagram $\pdgm_\homdim(f(\cdot, \bp))$ has a point with coordinates $(f(\sigma_b, \bp), f(\sigma_d, \bp))$ if $\dim(\sigma_b) = \homdim$. For every $\homdim$-dimensional simplex $\sigma_b$ that is unpaired in $P$, the diagram $\pdgm_\homdim(f(\cdot, \bp))$ contains the point $(f(\sigma_b, \bp), \infty)$. 
\end{enumerate}
In what follows, I elaborate on each step of the algorithm above. We focus on the case in which $\base$ is a triangulated surface.

\subsection{Special case: $\base$ is a triangulated surface}\label{sec:2D}
Let $\K$ be a \scom\ with simplices $\sigma_1, \ldots, \sigma_N$, indexed such that $i < j$ if $\sigma_i$ is a proper face of $\sigma_j$. Let $f: \K \times \base \to \mathbb{R}$ be a piecewise-linear fibered filtration function, and suppose that $\base$ is a triangulated surface. If $\tri$ is a triangle in $\base$, then $I(\sigma, \tau) \cap \tri$ is either $\emptyset$, $\tri$, a vertex of $\tri$, or a line segment whose endpoints are on $\partial \tri$. \Cref{fig:intersection} shows a few possible cases for $I(\sigma,\tau)$. The set in \Cref{eq:partition} is a set $L$ of line segments, and the polygonal subdivision induced by $L$ is a \emph{line arrangement}\footnote{Usually, a ``line arrangement'' refers to a planar subdivision that is induced by a set of lines in the plane. However, here I am using the term ``line arrangement'' slightly more generally. Every line segment in $L$ lies in some triangle $\tri$ of $\base$, and the line segments subdivide each triangle into polygons.} $\mathcal{A}(L)$. For an example of a line arrangement, see \Cref{fig:polygon partition}.

\begin{figure}
	\centering
	\subfloat[]{\includegraphics[width=.2\textwidth]{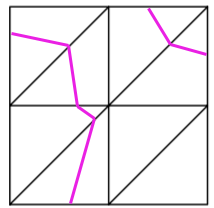}\label{fig:intersection 1}}
	\hspace{5mm}
	\subfloat[]{\includegraphics[width=.2\textwidth]{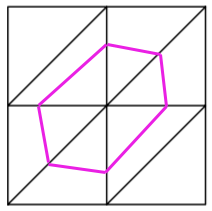}\label{fig:intersection 2}}
	\hspace{5mm}
	\subfloat[]{\includegraphics[width=.2\textwidth]{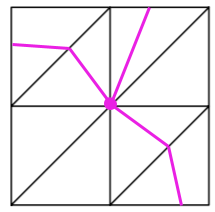}} \\
	\subfloat[]{\includegraphics[width=.2\textwidth]{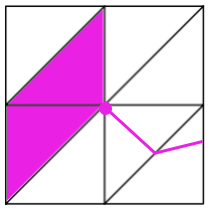}
	\label{fig:triangle_intersection}} 
	\hspace{5mm}
	\subfloat[]{\includegraphics[width=.2\textwidth]{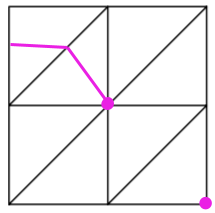}}
	\hspace{5mm}
	\subfloat[]{\includegraphics[width=.2\textwidth]{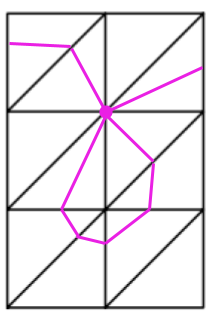}}
	\caption{A few possible cases for the set $I(\sigma,\tau)$, which is defined in \Cref{eq:Idef}. The black lines are the 1-skeleton of $\base$.}
	\label{fig:intersection}
\end{figure}

\begin{figure}
    \centering
    \includegraphics[width = .5\textwidth]{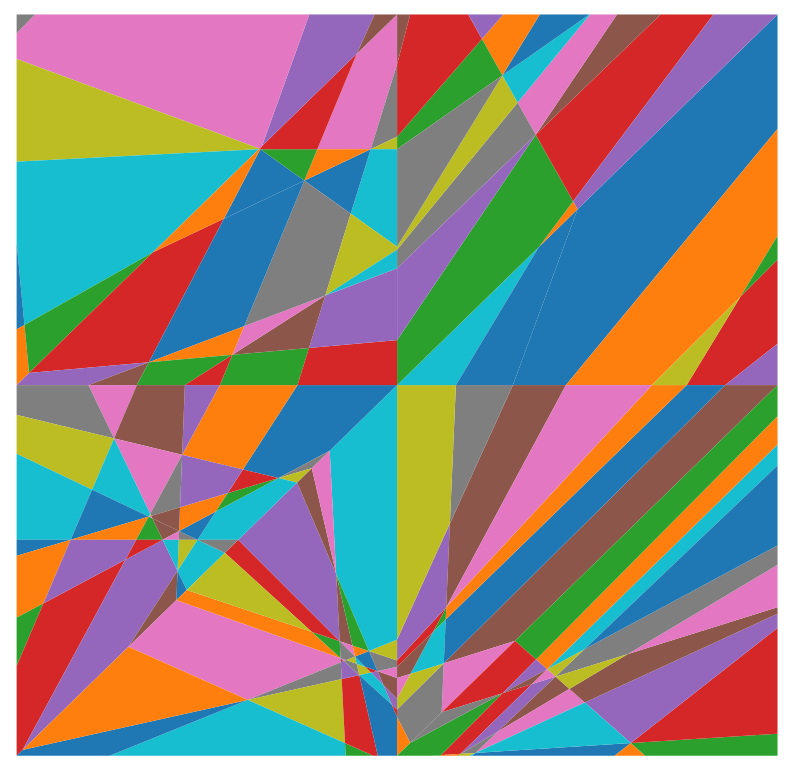}
    \caption{A line arrangement $\mathcal{A}(L)$ that represents the subdivision of a triangulated base space $\base$ into polygons (see \Cref{prop:polyhedrons_constant}). Within each polygon, the simplex indexing is constant. (This figure appeared originally in \cite{pd_bundle}.)}
    \label{fig:polygon partition}
\end{figure}

To simplify the exposition, we will make two genericity assumptions for the remainder of \Cref{sec:2D}. The idea of the algorithm is not different in the general case, but it requires some technical modifications, which we discuss in \Cref{sec:technical}. The assumptions are as follows:
\begin{enumerate}
    \item For all distinct simplices $\sigma, \tau \in \K$ and all vertices $v \in \base$, we have that $f(\sigma, v) \neq f(\tau, v)$. This implies that for all triangles $\tri \in \base$, the intersection $I(\sigma, \tau) \cap \tri$ is either $\emptyset$ or a line segment whose endpoints are not vertices of $\tri$. See the examples in Figures \ref{fig:intersection 1} and \ref{fig:intersection 2}.
    
    \item For all distinct simplices $\sigma_{i_1}, \sigma_{j_1}, \sigma_{i_2}, \sigma_{j_2} \in \K$ and every triangle $\tri \in \base$ such that $I(\sigma_{i_1}, \sigma_{j_1}) \cap \tri$ and $I(\sigma_{i_2}, \sigma_{j_2}) \cap \tri$ are nonempty, the line segments $I(\sigma_{i_1}, \sigma_{j_1}) \cap \tri$ and $ I(\sigma_{i_2}, \sigma_{j_2}) \cap \tri$ do not share any endpoints.
\end{enumerate}

Throughout \Cref{sec:2D}, we use the following notation. Let $m$ denote the number of triangles in $\base$. The numbers of vertices and edges in $\base$ are both $\mathcal{O}(m)$. Let $N$ denote the number of simplices in $\mathcal{K}$. For every triangle $\tri$ of $\base$, let $\kappa_\tri$ denote the number of line segments $\ell$ of the form $I(\sigma, \tau) \cap \tri$ and let $\kappa := \sum_{\tri \in \base} \kappa_\tri$. The worst case is $\kappa_\tri = \mathcal{O}(N^2)$ and $\kappa = \mathcal{O}(mN^2)$, but these are very crude upper bounds. Let $\mu_\tri$ denote the number of vertices of $\mathcal{A}(L)$ in the interior of triangle $\tri \in \base$; the quantity $\mu_\tri$ is equal to the number of points of the form $I(\sigma_{i_1}, \sigma_{j_1}) \cap I(\sigma_{i_2}, \sigma_{j_2}) \cap \tri$. Let $\mu$ denote the total number of vertices in $\mathcal{A}(L)$, which is $\sum_{\tri \in \base} \mu_\tri + \mathcal{O}(m + \kappa)$. In the worst case, $\mu_\tri = \mathcal{O}(\kappa_\tri^2) = \mathcal{O}(N^4)$ and $\mu = \mathcal{O}(mN^4)$, but these are again very crude upper bounds. The numbers of edges and polygons in $\mathcal{A}(L)$ are $\mathcal{O}(\mu)$.\footnote{Within each triangle $\tri$ of $\base$, the numbers of edges and polygons in $\mathcal{A}(L) \cap \tri$ are $\mathcal{O}(\mu_\tri)$ by Euler's formula, as noted in \cite{compgeo}.}


\subsubsection{Computing the polygons}\label{sec:findpolygons}
For a piecewise-linear vineyard, computing the intervals on which the simplex indexing is constant can be reduced to finding the intersections between the piecewise-linear functions $f(\sigma, t)$ and $f(\tau, t)$ for all pairs $(\sigma, \tau)$ of simplices in $\K$. Likewise, for a piecewise-linear PD bundle, computing the polygons on which the simplex indexing is constant can be reduced to finding the intersections $I(\sigma, \tau)$ for all pairs $(\sigma, \tau)$ of simplices.

We seek to compute the line arrangement $\mathcal{A}(L)$, where $L$ is the set of line segments defined by \Cref{eq:partition}. (See \Cref{fig:polygon partition}.) The polygons of $\mathcal{A}(L)$ are the polygons on which the simplex indexing is constant. We store $\mathcal{A}(L)$ using a doubly-connected edge list (DCEL) data structure, which is a standard data structure for storing a polygonal subdivision of the plane \cite{compgeo}. The DCEL data structure can be used without modification to represent $\mathcal{A}(L)$, which is a polygonal subdivision of a triangulated surface.\footnote{The primary reason to consider triangulated surfaces, rather than any simplicial complex $\base$ such that $\dim(\base) \leq 2$, is so that we can use a DCEL data structure to represent $\mathcal{A}(L)$. Otherwise, what follows in \Cref{sec:2D} works just as well for any $2$D simplicial complex.} The space complexity of $\mathcal{A}(L)$ is $\mathcal{O}(\mu)$. We compute $\mathcal{A}(L)$ using the following algorithm (illustrated in \Cref{fig:step1}):

\begin{figure}
    \centering
    \subfloat[]{\includegraphics[width = .25\textwidth]{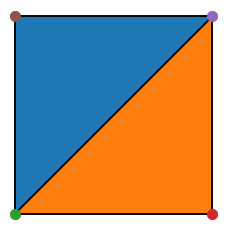}\label{fig:step1.1}}
    \subfloat[]{\includegraphics[width = .25\textwidth]{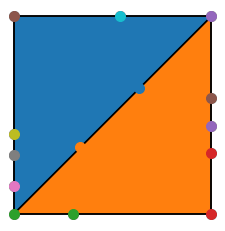}\label{fig:step1.2}} 
    \subfloat[]{\includegraphics[width = .25\textwidth]{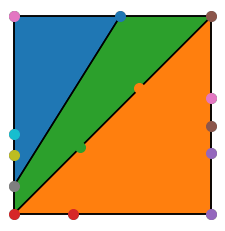}}
    \subfloat[]{\includegraphics[width = .25\textwidth]{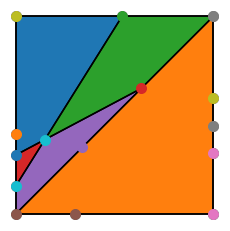}} \\
    \subfloat[]{\includegraphics[width = .25\textwidth]{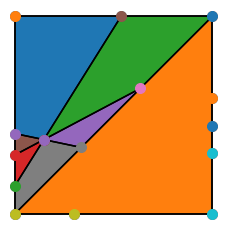}} 
    \subfloat[]{\includegraphics[width = .25\textwidth]{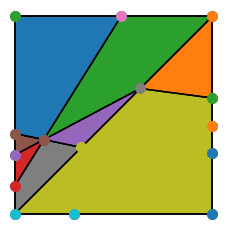}}
    \subfloat[]{\includegraphics[width = .25\textwidth]{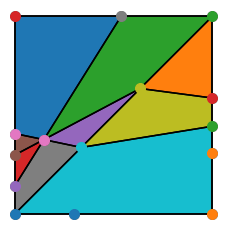}}
    \subfloat[]{\includegraphics[width = .25\textwidth]{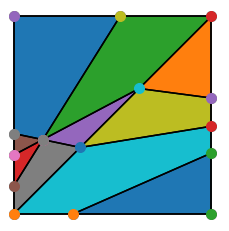}}
    \caption{Computing the polygons. (A) The line arrangement $\mathcal{A}(L)$ is initialized to represent the triangulated base space $\base$, which in this case consists of two triangles. (B) We find the vertices $v$ of $\mathcal{A}(L)$ that lie on the $1$-skeleton of $\base$. (C)--(H) For each triangle $\tri$ of $\base$, we incrementally add the line segments of the form $I(\sigma, \tau) \cap \tri$. The endpoints of a given line segment are a pair $(v, w)$ of vertices in (B). In (D), an internal vertex (a vertex at the intersection of two line segments) is created when the second line segment is added. Three of the line segments intersect at the internal vertex.}
    \label{fig:step1}
\end{figure}

\begin{enumerate}
\item We initialize $\mathcal{A}(L)$ so that it represents the triangulation $\base$. (See \Cref{fig:step1.1}.) In addition to the usual data that a DCEL stores, we enumerate the triangles in $\base$ and every half edge\footnote{An ``edge'' is a line segment in $\mathcal{A}(L)$ that connects two vertices $u$ and $v$. One can think of each edge as two \emph{half edges}, which are represented as the two oriented edges $u \to v$ and $v \to u$. If an edge is adjacent to two faces $F_1$ and $F_2$, then one of the half edges is associated with $F_1$ and the other is associated with $F_2$. One chooses an orientation (clockwise or counterclockwise) so that for every face $F$ in the DCEL, the half-edges associated with $F$ have this orientation (clockwise or counterclockwise) around the boundary of $F$.} $e$ stores the index of the triangle in $\base$ that $e$ is on the boundary of.

\vspace{3mm}

\item For every triangle $\tri \in \base$, we initialize an empty dictionary $\mathcal{D}(\tri)$.\footnote{A dictionary is a data structure for storing (key, value) pairs.} The keys will be pairs $(\sigma, \tau)$ of simplices for which $I(\sigma, \tau) \cap \tri$ is a line segment. The value of $(\sigma, \tau)$ will be a list of the endpoints $(v, w)$ of the line segment. We denote the value of $(\sigma, \tau)$ by $\mathcal{D}(\tri)[(\sigma, \tau)]$. These dictionaries use $\mathcal{O}(\kappa)$ space.

\vspace{3mm}

\item For each edge $e$ of $\base$, we compute the vertices of $\mathcal{A}(L)$ that lie on $e$. (See \Cref{fig:step1.2}.) These are the vertices that equal $I(\sigma, \tau) \cap e$ for some pair $(\sigma, \tau)$ of simplices. To do this, we consider the restriction of $f$ to $e$; the restriction $f|_e$ is a $1$-parameter filtration function (the input to a vineyard). For each $\sigma \in \K$, the set $\{(\bp, f(\sigma, \bp)) \mid \bp \in e\}$ is a line segment $\ell_{\sigma, e}$ and $I(\sigma, \tau) \cap e$ is the point $\bp \in e$ at which $\ell_{\sigma, e}$ and $\ell_{\tau, e}$ intersect. We use the Bentley--Ottman planesweep algorithm \cite{compgeo} to compute these intersections, thus obtaining the vertices $v$ of $\mathcal{A}(L)$ that lie on $e$. For a vertex $v$ that equals $I(\sigma, \tau) \cap e$, we add $v$ to the list $\mathcal{D}(\tri)[(\sigma, \tau)]$ for each triangle $\Delta$ to which $e$ is adjacent. Completing step 3 takes $\mathcal{O}(N)$ space and $\mathcal{O}((mN + \kappa) \log N)$ time in total for all edges in $\base$, and can be parallelized over the edges.\footnote{In some cases, in may be more efficient to consider the restriction of $f$ to an Euler path $\gamma$ through the $1$-skeleton of $\base$, rather than the restriction of $f$ to each edge separately. For example, if $\base$ is of the form in \Cref{fig:triangulation grid}, then an Euler path is given in \Cref{fig:triangulation path}.}

\begin{figure}
    \centering
    \subfloat[]{\includegraphics[width = .35\textwidth]{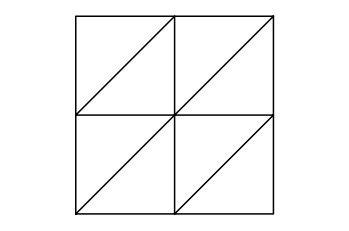}\label{fig:triangulation grid}}
    \subfloat[]{\includegraphics[width = .35\textwidth]{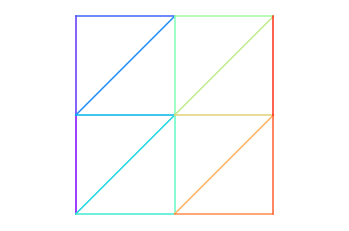}\label{fig:triangulation path}}
    \caption{(A) A triangulated base space $\base$. (B) An Euler path $\gamma$ through the $1$-skeleton of $\base$, starting at the bottom-left vertical edge (violet) and ending at the top-right vertical edge (red).}
\end{figure}

\vspace{3mm}

\item For each triangle $\tri \in \base$ and for each key $(\sigma, \tau)$ in the dictionary $\mathcal{D}(\tri)$, there is an associated pair $(v, w)$ of vertices that are the endpoints of the line segment $I(\sigma, \tau) \cap \tri$. We seek to add all of these line segments to the DCEL that represents $\mathcal{A}(L)$. (See Figures \ref{fig:step1}C--H.) There are many standard algorithms for doing this; one example is the incremental algorithm (see, e.g., Chapter 8.3 of \cite{compgeo}), in which the line segments are incrementally added one at a time. The worst-case run time of the incremental algorithm in triangle $\tri$ is $\mathcal{O}( \kappa_\tri^2)$, which yields a total run time of $\mathcal{O}(\sum_{\tri \in \base} \kappa_\tri^2)$. For better performance, this algorithm can be parallelized over the triangles $\tri \in \base$.
\end{enumerate}

\noindent In \Cref{fig:step1}, we illustrate the algorithm for computing the polygons.

As we add line segments to $\mathcal{A}(L)$, we keep track of the pairs $(\sigma, \tau)$ of simplices that correspond to each edge of $\mathcal{A}(L)$. If edge $e$ of $\mathcal{A}(L)$ is a subset of $I(\sigma, \tau)$, then $e$ stores a reference to the pair $(\sigma, \tau)$. We add the reference to $(\sigma, \tau)$ at the time that edge $e$ is created in $\mathcal{A}(L)$. If $P_1$, $P_2$ are adjacent polygons of $\mathcal{A}(L)$ that share edge $e$, then the simplex indexings in $P_1$, $P_2$ are related via the transposition of $\sigma$ and $\tau$.

\subsubsection{Computing the pairing function}\label{sec:simplexpairs}
Let $G$ be the dual graph of the line arrangement $\mathcal{A}(L)$. The graph $G$ has a vertex $v_P$ for every polygon $P$ of $\mathcal{A}(L)$ and an edge between two vertices if the corresponding polygons are adjacent. We compute a path $\Gamma$ that visits every vertex of $G$ at least once. For an example, see \Cref{fig:path}. One way to obtain such a path is via depth-first search, which takes $\mathcal{O}(\mu)$ time because the number of nodes in $G$ (i.e., polygons of $\mathcal{A}(L)$) is $\mathcal{O}(\mu)$. This yields a path $\Gamma$ whose length is $\mathcal{O}(\mu)$.

\begin{figure}
    \centering
    \includegraphics[width = .3\textwidth]{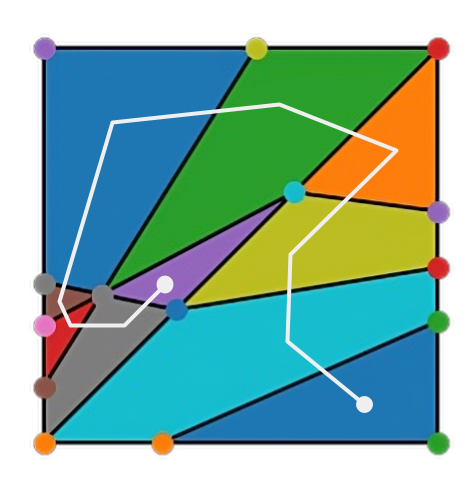}
    \caption{A path $\Gamma$ (in white) that visits every polygon in the line arrangement $\mathcal{A}(L)$.}
    \label{fig:path}
\end{figure}

At the first vertex $v_P$ of $\Gamma$, we compute the simplex indexing in polygon $P$, the RU decomposition of the boundary matrix in $P$, and the (birth, death) simplex pairs in $P$. To store the RU decomposition, we use the sparse matrix data structure from \cite{vineyards}. The polygon $P$ stores a reference to its (birth, death) simplex pairs. To store the current simplex indexing, each simplex stores a reference to its index in the current indexing (which we initialize to the indexing in $P$).

We traverse the path $\Gamma$. As we walk from one polygon $P_1$ to the next polygon $P_2$ by crossing an edge $e$ in $\mathcal{A}(L)$, we update the simplex indexing, the RU decomposition, and the (birth, death) simplex pairs. To update the simplex indexing, recall that edge $e$ stores a reference to the simplex pair $(\sigma, \tau)$ such that $e \subseteq I(\sigma, \tau)$. This implies that the simplex indexings in $P_1$ and $P_2$ are related via the transposition of $\sigma$ and $\tau$ because we must have (without loss of generality) $f(\sigma, \bp) > f(\tau, \bp)$ for $\bp \in P_1$ and $f(\sigma, \bp) < f(\tau, \bp)$ for $\bp \in P_2$, with $f(\sigma, \bp) = f(\tau, \bp)$ on the shared edge $e$. We update the simplex indexing by swapping the indices that $\sigma$ and $\tau$ store. To update the RU decomposition and the (birth, death) simplex pairs, we apply the update algorithm of \cite{vineyards}. This takes $\mathcal{O}(N)$ time in the worse case, but often approximately constant time in practice; see \cite{vineyards} and the earlier discussion in \Cref{sec:vineyards}. In $P_2$, we store the new (birth, death) simplex pairs. Storing the simplex pairs uses $\mathcal{O}(N)$ space for each polygon of $\mathcal{A}(L)$, so we use $\mathcal{O}(N\mu)$ space in total.

We can optimize the space requirements by recalling from \cite{vineyards} that most updates of the simplex indexing do not change the (birth, death) simplex pairs.\footnote{However, we note that we do not know in advance whether an update of the simplex indexing will change the (birth, death) simplex pairs.} If the update from $P_1$ to $P_2$ does not change the simplex pairs, we can delete the edge in $\mathcal{A}(L)$ that $P_1$ and $P_2$ share, thus merging $P_1$ and $P_2$ into a single polygon and reducing the size of $\mathcal{A}(L)$.

\subsubsection{Querying the PD bundle}\label{sec:query}
We consider the scenario in which a user seeks to query many points $\bp \in \base$ in real time and see the $\homdim$th persistence diagram $\pdgm_\homdim(f(\cdot, \bp))$ associated with each queried point $\bp$.

To compute the $\homdim$th persistence diagram $\pdgm_\homdim(f(\cdot, \bp))$ associated with a given $\bp$, we first identify the polygon $P$ of $\mathcal{A}(L)$ that contains $\bp$. We do this in two steps. The first step is to identify the triangle $\tri \in \base$ that contains $\bp$. This takes $\mathcal{O}(m)$ worst-case time because it takes constant time to test if a given triangle contains $\bp$. In certain cases, one can identify the triangle $\tri$ more efficiently. For example, if $\base$ is a triangulation of the form in \Cref{fig:triangulation grid}, then one can locate the triangle $\tri$ in $\mathcal{O}(1)$ time by simply examining the coordinates of $\bp$. The second step is to locate the polygon $P$ in $\tri$ that contains $\bp$. This is a well-studied problem in computational geometry; it is known as the \emph{point-location problem} \cite{handbook}. When one is planning to perform many point-location queries on the same line arrangement (i.e., if one is querying many points $\bp \in \base$), the standard strategy is to precompute a data structure so that the subsequent point-location queries can be done efficiently. There are many strategies for doing this (see, e.g., Chapter 38 in \cite{handbook}). One method is the slab-and-persistence method \cite{slab}, in which one precomputes a ``persistent search tree'' for the line arrangement.\footnote{If $\base$ is a 2D triangulated subset of $\mathbb{R}^2$, then one can also use the slab-and-persistence method to perform the first step---locating the triangle $\tri \in \base$ that contains $\bp$.} We construct a persistent search tree for each triangle in $\base$. Using separate persistent search trees for the planar subdivisions in each triangle, the slab and persistence method takes $\mathcal{O}(\sum_{\tri \in \base} \mu_\tri \log(\mu_\tri))$ preprocessing time, $\mathcal{O}(\mu)$ space, and $\mathcal{O}(\max_{\tri \in \base} \log \mu_\tri)$ time per query.

We obtain $\pdgm_\homdim(f(\cdot, \bp))$ by evaluating $f(\cdot, \bp)$ on the simplices in the pairing function for polygon $P$, which was precomputed in the previous step (see \Cref{sec:simplexpairs}). This takes $\mathcal{O}(N)$ time. For every (birth, death) pair $(\sigma_b, \sigma_d)$ of simplices, $\pdgm_\homdim(f(\cdot, \bp))$ has a point with coordinates $(f(\sigma_b, \bp), f(\sigma_d, \bp))$ if $\dim(\sigma_b) = \homdim$. For every unpaired $\homdim$-dimensional simplex $\sigma_b$, the diagram $\pdgm_\homdim(f(\cdot, \bp))$ has a point with coordinates $(f(\sigma_b, \bp), \infty)$.

\subsection{Generalizing to higher-dimensional base spaces}\label{sec:high}

The algorithm of \Cref{sec:2D} (as outlined at the beginning of \Cref{sec:compute}) does not require many modifications for higher-dimensional base spaces $\base$. We replace the subdivision of $\base$ into polygons by a subdivision of $\base$ into $n$-dimensional polyhedra, where $n = \dim(\base)$. In the third step (querying the PD bundle), one uses a point-location algorithm for a hyperplane arrangement (see, e.g., \cite{meiser, chazelle}). Only the first step (computing the polyhedra) requires a meaningful modification, which I describe below.

When $n = 2$, the intersection of $I(\sigma, \tau)$ with a triangle $\tri \in \base$ is the intersection of a line with $\tri$, which is a line segment $L_{\sigma, \tau, \tri}$. These line segments completely determine the polygonal subdivision of $\base$ because the line segments are the faces of the polygons. In turn, each line segment $L_{\sigma, \tau, \tri}$ is completely determined by the intersection of $L_{\sigma, \tau, \tri}$ with the $1$-skeleton of $\base$; this intersection is a pair $(v_{\sigma, \tau, \tri}, w_{\sigma, \tau, \tri})$ of points. In \Cref{sec:findpolygons}, we computed the set \{$(v_{\sigma, \tau, \tri}, w_{\sigma, \tau, \tri})\}_{\sigma, \tau, \tri}$ by restricting the fibered filtration function $f$ to each edge of $\base$ and applying the Bentley--Ottman planesweep algorithm.
    
In general, the intersection of $I(\sigma, \tau)$ with an $n$-simplex $\Delta \in \base$ is the intersection of an $(n-1)$-dimensional hyperplane $H_{\sigma, \tau, \Delta}$ with $\Delta$; the intersection is an $(n-1)$-dimensional polyhedron $P_{\sigma, \tau, \Delta}$. The set $\{P_{\sigma, \tau, \Delta}\}_{\sigma, \tau, \Delta}$ completely determines the polyhedral subdivision of $\base$ that is given by \Cref{prop:polyhedrons_constant} because the polyhedra $P_{\sigma, \tau, \Delta}$ are the $(n-1)$-dimensional faces of the $n$-dimensional polyhedra in the subdivision. In turn, each polyhedron $P_{\sigma, \tau, \Delta}$ is completely determined by its intersection with the edges of $\base$, as follows. The $m$-dimensional faces of $P_{\sigma, \tau, \Delta}$ are the set 
\begin{equation*}
    \{H_{\sigma, \tau, \Delta} \cap \Delta^{(m+1)} \mid \Delta^{(m+1)} \text{ is an } (m+1)\text{-dimensional face of } \Delta \text{ and } H_{\sigma, \tau, \Delta} \cap \Delta^{(m+1)} \neq \emptyset \}\,.
\end{equation*}
For every $(m+1)$-dimensional face $\Delta^{(m+1)}$ of $\Delta$ such that $H_{\sigma, \tau, \Delta} \cap \Delta^{(m+1)} \neq \emptyset$, the intersection $H_{\sigma, \tau, \Delta} \cap \Delta^{(m+1)}$ is the $m$-dimensional polyhedron whose $(m-1)$-dimensional faces are the set
\begin{equation*}
    \{H_{\sigma, \tau, \Delta} \cap \Delta^{(m)} \mid \Delta^{(m)} \text{ is an } m\text{-dimensional face of } \Delta^{(m+1)}\}\,.
\end{equation*}
By induction, the faces of $P_{\sigma, \tau, \Delta}$ are determined by 
\begin{equation*}
    \{H_{\sigma, \tau, \Delta} \cap e \mid e \text{ is a } 1 \text{-dimensional face of } \Delta \text{ (i.e., $e$ is an edge)}\}\,,
\end{equation*}
which are the vertices of $P_{\sigma, \tau, \Delta}$. Consequently, we can compute each $P_{\sigma, \tau, \Delta}$ by determining its vertices. As in the case in which $\base$ is a triangulated surface, we do this by restricting the fibered filtration function $f$ to each edge of $\base$ and applying the Bentley--Ottman planesweep algorithm.
\section{Conclusions and Discussion}\label{sec:conclusion}
I introduced an algorithm for efficiently computing persistence diagram (PD) bundles when the fibered filtration function is piecewise linear. I gave full details for the case in which the base space $\base$ is a triangulated surface. Additionally, in Section \ref{sec:high}, I discussed how one can generalize the algorithm to higher dimensions. I conclude with some questions and proposals for future work:
\begin{itemize}
    \item What invariants can we use to summarize and analyze PD bundles in ways that do not require exploratory data analysis?

    \vspace{3mm}
    
    \noindent The current algorithm asks a user to ``query'' the PD bundle at various points in the base space. This is useful for qualitative analysis or if one has a function whose input is a PD and one seeks to optimize that function over $\base$. However, other applications may require global invariants.

    \vspace{3mm} 
    
    \item How do we generalize the algorithm to fibered filtration functions that are not piecewise linear?

    \vspace{3mm}
    
    \noindent For piecewise-linear fibered filtration functions, we used the fact that the base space $\base$ can be subdivided into polyhedrons such that there is a single PD ``template'' (a list of (birth, death) simplex pairs) for each polyhedron. The template can then be used to obtain $\pdgm_\homdim(f_\bp)$ at any point $\bp$ in the polyhedron. For ``generic'' fibered filtration functions, it was shown in \cite{pd_bundle} that the base space $\base$ is stratified such that for each stratum, there is a single PD template that can be used to obtain $\pdgm_\homdim(f_\bp)$ at any point $\bp$ in the stratum.

    \vspace{3mm}
    
    \item How do we generalize to the case where $\K^{\bp}$ is not constant with respect to $\bp \in \base$?

    \vspace{3mm}
    
    \noindent In this case, simplices are added and removed from the filtration as $\bp \in \base$ varies, so the algorithm must be modified.
\end{itemize}

\section*{Acknowledgements}
I thank Michael Lesnick and Nina Otter for helpful discussions about computing PD bundles.

\appendix
\section*{Appendix}
\renewcommand{\thesection}{A} 

\subsection{Technical details of the algorithm}\label{sec:technical}
Let $\K$ be a simplicial complex with $N$ simplices, indexed $\sigma_1, \ldots, \sigma_N$ such that $i < j$ if $\sigma_i$ is a proper face of $\sigma_j$. Let $\base$ be a triangulated surface, and let $f: \K \times \base \to \mathbb{R}$ be a piecewise-linear fibered filtration function. In \Cref{sec:2D}, we made two generic assumptions to simplify the exposition. If assumption (1) holds, then every nonempty $I(\sigma_i, \sigma_j) \cap \tri$ is a line segment $\ell$ that subdivides triangle $\tri$ into polygons $Q_1$, $Q_2$ such that
\begin{equation*}
    (\idx_f(\sigma_i, \bp_1) - \idx_f(\sigma_j, \bp_1)) \times (\idx_f(\sigma_i, \bp_2) - \idx_f(\sigma_j, \bp_2)) < 0
\end{equation*}
for all $\bp_1 \in Q_1$ and $\bp_2 \in Q_2$ (i.e., $\sigma_i$ and $\sigma_j$ have different relative orders in $Q_1$ and $Q_2$). We say that $\sigma_i$ and $\sigma_j$ \emph{swap along $\ell$} because $\sigma_i$ and $\sigma_j$ have different relative orders on either side of $\ell$. If assumption (1) does not hold, then for any triangle $\tri$ in $\base$ and pair $(\sigma_i, \sigma_j)$ of simplices, it is possible that $I(\sigma_i, \sigma_j) \cap \tri$ equals either $\tri$ or an edge of $\tri$. If $e$ is an edge of triangle $\tri$ such that $e \subseteq I(\sigma_i, \sigma_j) \cap \tri$, then $\sigma_i$ and $\sigma_j$ \emph{swap along line segment $e$} if 
\begin{equation*}
    (\idx_f(\sigma_i, \bp_1) - \idx_f(\sigma_j, \bp_1)) \times (\idx_f(\sigma_i, \bp_2) - \idx_f(\sigma_j, \bp_2)) < 0
\end{equation*}
for all $\bp_1 \in \tri_1$ and $\bp_2 \in \tri_2$, where $\tri_1$ and $\tri_2$ are the triangles of $\base$ that are adjacent to $e$ (i.e., $\sigma_i$ and $\sigma_j$ have different relative orders in $\tri_1$ and $\tri_2$). In \Cref{fig:generalcase}, we illustrate an example where assumption (1) does not hold. We highlight the line segments that $\sigma_i$ and $\sigma_j$ swap along.

If assumption (2) does not hold, then there may be a line segment $\ell$ in a triangle $\tri$ such that $I(\sigma_{i_1}, \sigma_{j_1}) \cap \tri = \ell = I(\sigma_{i_2}, \sigma_{j_2}) \cap \tri$ for two distinct pairs $(\sigma_{i_1}, \sigma_{j_1})$ and $(\sigma_{i_2}, \sigma_{j_2})$ of simplices in $\K$.

\begin{figure}
    \centering
    \subfloat[]{\includegraphics[width = .3\textwidth]{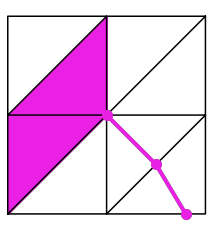}}
    \hspace{3mm}
    \subfloat[]{\includegraphics[width=.3\textwidth]{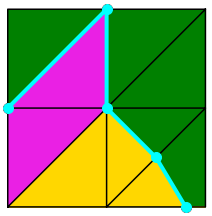}}
    \hspace{3mm}
    \subfloat[]{\includegraphics[width = .3\textwidth]{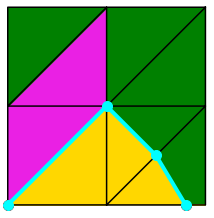}}
\caption{Examples of fibered filtration functions for which assumption (1) of \Cref{sec:2D} does not hold. (A) The pair $(\sigma_i, \sigma_j)$ is a pair of simplices such that $I(\sigma_i, \sigma_j) \cap \tri = \tri$ for every pink triangle $\tri$ and $I(\sigma_i, \sigma_j) \cap \tri = \ell$ if $\ell \subseteq \tri$ is a pink line segment. Without loss of generality $i < j$, so $\idx_f(\sigma_i, \bp) < \idx_f(\sigma_j, \bp)$ if $f(\sigma_i, \bp) = f(\sigma_j, \bp)$. (B) Suppose that $f(\sigma_j, \bp) < f(\sigma_i, \bp)$ on green triangles and $f(\sigma_i, \bp) < f(\sigma_j, \bp)$ on yellow triangles. In blue, we draw the line segments on which the pair $(\sigma_i, \sigma_j)$ swaps. (C) Suppose that $f(\sigma_i, \bp) < f(\sigma_j, \bp)$ on green triangles and $f(\sigma_j, \bp) < f(\sigma_i, \bp)$ on yellow triangles. In blue, we again draw the line segments on which the pair $(\sigma_i, \sigma_j)$ swaps.}
    \label{fig:generalcase} 

\end{figure}

In Sections \ref{sec:step1mods} and \ref{sec:step2mods}, I explain the modifications for the algorithm when we do not make the assumptions of \Cref{sec:2D}. It suffices to modify step 1 (see \Cref{sec:findpolygons}) and step 2 (see \Cref{sec:simplexpairs}). Step 3 (\Cref{sec:query}) remains the same.

\subsubsection{Preliminaries}

As in \Cref{sec:findpolygons}, we consider the restriction of $f$ to every edge $e$ of $\base$ to find the vertices of $\mathcal{A}(L)$ that lie on the $1$-skeleton of $\base$.

\begin{definition}
\noindent A vertex $v$ for a pair $(\sigma, \tau)$ of simplices is \emph{detected along edge $e$} if, while traversing edge $e$ during the Bentley--Ottman algorithm, we detect the point $v \in e$ as a point at which the relative order of $\sigma$ and $\tau$ changes.
\end{definition}
\noindent A vertex $v$ for the pair $(\sigma, \tau)$ is detected along edge $e$ if and only if $\sigma$ and $\tau$ have different relative orders at the endpoints of $e$.

\begin{definition}\label{def:edge_detection}
A line segment $(v, w)$ is \emph{detected in triangle $\tri$} if there is a pair $(\sigma, \tau)$ of simplices such that vertex $v$ is detected along an edge $e_1$ of $\tri$ for $(\sigma, \tau)$ and vertex $w$ is detected along an edge $e_2$ of $\tri$ for $(\sigma, \tau)$.
\end{definition}
\Cref{lem:linesegs} characterizes the conditions for which a pair of simplices swaps along a line segment.

\begin{lemma}\label{lem:linesegs}
Let $(v, w)$ be a line segment that is not on the boundary of $\base$.
\begin{enumerate}
\item If $v$ and $w$ are not the endpoints of an edge in $\base$, let $\tri$ be the unique triangle that contains $(v, w)$. A pair $(\sigma_i, \sigma_j)$ of simplices swaps along the line segment $(v, w)$ if and only if $(v, w)$ is detected in triangle $\tri$. 

\item If $v$ and $w$ are the endpoints of an edge in $\base$, let $\tri_1$, $\tri_2$ be the two triangles that are adjacent to that edge. A pair $(\sigma_i, \sigma_j)$ of simplices swaps along the line segment $(v, w)$ if and only if $(v, w)$ is detected in exactly one of $\tri_1$, $\tri_2$.
\end{enumerate}
\end{lemma}
\begin{proof}
Statement (1) was the situation in \Cref{sec:2D}, so it remains only to prove statement (2).

Suppose that $(v, w)$ are the endpoints of an edge $e$ in $\base$ that is not on the boundary of $\base$. Let $\tri_1$, $\tri_2$ be the two triangles that are adjacent to $\tri$. As illustrated in \Cref{fig:subcase_labels}, we denote the third vertex of $\tri_1$ by $u_1$, the third vertex of $\tri_2$ by $u_2$, the other two edges in $\tri_1$ by $e_2$, $e_3$, and the other two edges in $\tri_2$ by $e_4$, $e_5$.

A pair $(\sigma_i, \sigma_j)$ swaps along $(v, w)$ only if $e \subseteq I(\sigma_i, \sigma_j) \cap \tri_1 \cap \tri_2$. If $e \subseteq I(\sigma_i, \sigma_j) \cap \tri$ for triangle $\tri$, then either $I(\sigma_i, \sigma_j) \cap \tri = e$ or $I(\sigma_i, \sigma_j) \cap \tri = \tri$. If both $I(\sigma_i, \sigma_j) \cap \tri_1 = \tri_1$ and $I(\sigma_i, \sigma_j) \cap \tri_2 = \tri_2$, then $(\sigma_i, \sigma_j)$ does not swap along $(v, w)$ because $\sigma_i$ and $\sigma_j$ have the same relative order in $\tri_1$ and $\tri_2$. Therefore, the pair $(\sigma_i, \sigma_j)$ swaps along $(v, w)$ only if the intersection of $I(\sigma_i, \sigma_j)$ with one triangle is $e$ and the intersection with the other triangle is either $e$ or the entire triangle. Without loss of generality, $I(\sigma_i, \sigma_j) \cap \tri_1 = e$ and either $I(\sigma_i, \sigma_j) \cap \tri_2 = \tri_2$ or $I(\sigma_i, \sigma_j) \cap \tri_2 = e$.

Similarly, the line segment $(v, w)$ can be detected in triangle $\tri_k$ only if $e = I(\sigma_i, \sigma_j) \cap \tri_k$. If $I(\sigma_i, \sigma_j) \cap \tri_1 = e$, then we must also have $e \subseteq I(\sigma_i, \sigma_j) \cap \tri_2$, so either $I(\sigma_i, \sigma_j) \cap \tri_2 = \tri_2$ or $I(\sigma_i, \sigma_j) \cap \tri_2 = e$ (and vice versa if $I(\sigma_i, \sigma_j) \cap \tri_2 = e$). Therefore, $(v, w)$ is detected in $\tri_k$ only if $I(\sigma_i, \sigma_j) = e$ and the intersection with the other triangle is either $e$ or the whole triangle. Without loss of generality, $\tri_k = \tri_1$.

In Figures \ref{fig:subcase_1.1}--\ref{fig:subcase_2.3}, we illustrate the possible cases in which $I(\sigma_i, \sigma_j) \cap \tri_1 = e$ and either $I(\sigma_i, \sigma_j) \cap \tri_2 = \tri_2$ or $I(\sigma_i, \sigma_j) \cap \tri_2 = e$. We will show that statement (2) holds in each of these cases. In all other cases, we have already shown that neither $(\sigma_i, \sigma_j)$ swaps along the line segment $(v, w)$ nor is $(v, w)$ detected in $\tri_1$ or $\tri_2$.

\begin{figure}
    \centering
    \subfloat[]{\includegraphics[width=.3\textwidth]{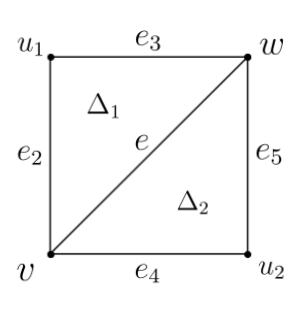}\label{fig:subcase_labels}} 
    \subfloat[Case 1.1]{\includegraphics[width=.3\textwidth]{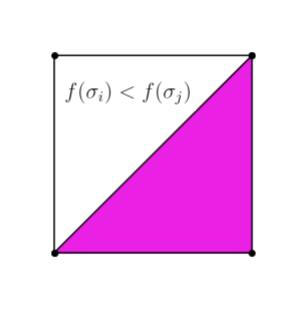}\label{fig:subcase_1.1}}
    \subfloat[Case 1.2]{\includegraphics[width=.3\textwidth]{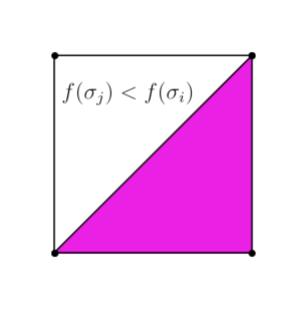}\label{fig:subcase_1.2}} \\
    \subfloat[Case 2.1]{\includegraphics[width=.3\textwidth]{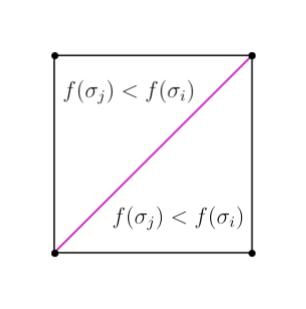}\label{fig:subcase_2.1}}
    \subfloat[Case 2.2]{\includegraphics[width=.3\textwidth]{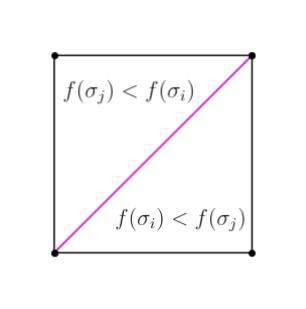}\label{fig:subcase_2.2}} \subfloat[Case 2.3]{\includegraphics[width=.3\textwidth]{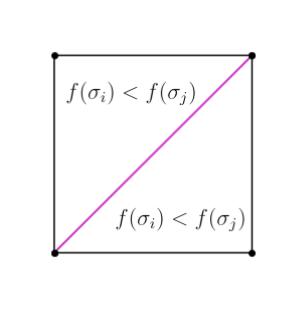}\label{fig:subcase_2.3}}
    \caption{(A) The vertices, edges, and triangles that were defined in the proof of \Cref{lem:linesegs}. (B--F) The cases in the proof of \Cref{lem:linesegs}. Pink regions are regions on which $\sigma_i$ and $\sigma_j$ have equal filtration values.}
    \label{fig:subcases}
\end{figure}

Without loss of generality, we assume that $i < j$, so $\idx_f(\sigma_i, \bp) < \idx_f(\sigma_j, \bp)$ if $\bp \in I(\sigma_i, \sigma_j)$. 

\vspace{3mm}

{\bf Case 1:} $I(\sigma_i, \sigma_j) \cap \tri_2 = \tri_2$.

\vspace{3mm}

There are two subcases.
\begin{enumerate}
\item {\bf Case 1.1: } (see \Cref{fig:subcase_1.1}) $f(\sigma_i, \bp) < f(\sigma_j, \bp)$ for all $\bp \in \tri_1 \setminus e$.

\vspace{3mm}

\noindent In this subcase, $\idx_f(\sigma_i, \bp) < \idx_f(\sigma_j, \bp)$ for all $\bp \in \tri_1 \cup \tri_2$. Therefore, the pair $(\sigma_i, \sigma_j)$ does not swap along $(v, w)$. Neither $v$ nor $w$ is detected along any edges of $\tri_1$ or $\tri_2$, so the line segment $(v, w)$ is not detected in either $\tri_1$ or $\tri_2$.
\vspace{3mm}

\item {\bf Case 1.2: } (see \Cref{fig:subcase_1.2}) $f(\sigma_j, \bp) < f(\sigma_i, \bp)$ for all $\bp \in \tri_1 \setminus e$.

\vspace{3mm}

\noindent In this subcase, $\idx_f(\sigma_i, \bp) < \idx_f(\sigma_j, \bp)$ for all $\bp \in \tri_2$ and $\idx_f(\sigma_j, \bp) < \idx_f(\sigma_i, \bp)$ for all $\bp \in \tri_1 \setminus e$. Therefore, the pair $(\sigma_i, \sigma_j)$ swaps along $(v, w)$. The vertex $v$ is detected along edge $e_2$ and the vertex $w$ is detected along the edge $e_3$. Because $e_2$ and $e_3$ are edges of $\tri_1$, the line segment $(v, w)$ is detected in $\tri_1$. The vertices $v$ and $w$ are not detected along any edge of $\tri_2$, so $(v, w)$ is not detected in $\tri_2$.
\end{enumerate}

{\bf Case 2:} $I(\sigma_i, \sigma_j) \cap \tri_2 = e$.
\vspace{3mm}

There are three subcases.
\begin{enumerate}
\item {\bf Case 2.1:} (see \Cref{fig:subcase_2.1}) $f(\sigma_j, \bp) < f(\sigma_i, \bp)$ for all $\bp \in (\tri_1 \cup \tri_2) \setminus e$.
\vspace{3mm}

\noindent In this subcase, $\idx_f(\sigma_j, \bp) < \idx_f(\sigma_i, \bp)$ for all $\bp \in (\tri_1 \cup \tri_2)\setminus e.$ Therefore, the pair $(\sigma_i, \sigma_j)$ does not swap along $(v, w)$. The vertex $w$ is detected along edges $e_3$ and $e_5$. The vertex $v$ is detected along edges $e_2$ and $e_4$. Therefore, $(v, w)$ is detected in both $\tri_1$ and $\tri_2$.

\vspace{3mm}

\item {\bf Case 2.2:} (see \Cref{fig:subcase_2.2}) Either we have $f(\sigma_j, \bp) < f(\sigma_i, \bp)$ for all $\bp \in \tri_1 \setminus e$ and $f(\sigma_i, \bp) < f(\sigma_j, \bp)$ for all $\bp \in \tri_2 \setminus e$ or we have $f(\sigma_i, \bp) < f(\sigma_j, \bp)$ for all $\bp \in \tri_1 \setminus e$ and $f(\sigma_j, \bp) < f(\sigma_i, \bp)$ for all $\bp \in \tri_1 \setminus e$. Without loss of generality, we assume the former.

\vspace{3mm}

\noindent In this subcase, $\idx_f(\sigma_i, \bp) < \idx_f(\sigma_j, \bp)$ for all $\bp \in \tri_2$ and $\idx_f(\sigma_j, \bp) < \idx_f(\sigma_i, \bp)$ for all $\bp \in \tri_1 \setminus e$. Therefore, the pair $(\sigma_i, \sigma_j)$ swaps along $(v, w)$. The vertex $v$ is detected along $e_2$ and the vertex $w$ is detected along $e_3$, so $(v, w)$ is detected in triangle $\tri_1$. Neither $v$ nor $w$ is detected along any edge of $\tri_2$, so $(v, w)$ is not detected in $\tri_2$.

\vspace{3mm}

\item {\bf Case 2.3:} (see \Cref{fig:subcase_2.3}) $f(\sigma_i, \bp) < f(\sigma_j, \bp)$ for all $\bp \in (\tri_1 \cup \tri_2) \setminus e$.
\vspace{3mm}

\noindent In this subcase, $\idx_f(\sigma_i, \bp) < \idx_f(\sigma_j, \bp)$ for all $\bp \in \tri_1 \cup \tri_2$. Therefore, the pair $(\sigma_i, \sigma_j)$ does not swap along $(v, w)$. Neither $v$ nor $w$ is detected along any edge of $\tri_1$ or $\tri_2$, so $(v, w)$ is not detected in either $\tri_1$ or $\tri_2$.
\end{enumerate}
\end{proof}

We use \Cref{lem:multiple_trans} to modify step 2 of the algorithm: computing the simplex pairing function.

\begin{lemma}\label{lem:multiple_trans}
Let $\idx_0, \idx_1 : \K \to \{1, \ldots, N\}$ denote two different simplex indexings (not necessarily compatible with $f$; see \Cref{def:spx_idx_general}), where $N$ is the number of simplices in $\K$. Let $\{(\sigma_{i_k}, \sigma_{j_k})\}_{k=1}^m$ be the set of pairs $(\sigma_{i_k}, \sigma_{j_k})$ such that
\begin{equation*}
    (\idx_0(\sigma_{i_k})- \idx_0(\sigma_{j_k}))\times(\idx_1(\sigma_{i_k}) - \idx_1(\sigma_{j_k})) < 0\,.
\end{equation*}
That is, $\sigma_{i_k}$ and $\sigma_{j_k}$ have different relative orders in $\idx_0$ and $\idx_1$. Let $\zeta_0:= \idx_0$, and for $k \in \{1, \ldots, m\}$, let $\zeta_k : \K \to \{1, \ldots, N\}$ be the simplex indexing obtained by transposing $(\sigma_{i_k}, \sigma_{j_k})$ in the simplex indexing $\zeta_{k-1}$. If $\zeta_{k-1}(\sigma_{i_k})$ and $\zeta_{k-1}(\sigma_{j_k})$ are consecutive integers for all $k$, then $\zeta_m = \idx_1$. Furthermore, the sequence $\{(\sigma_{i_k}, \sigma_{j_k})\}_{k=1}^m$ can be ordered so that this conditions holds.
\end{lemma}
\begin{proof}
First, we prove that there is at least one pair $(\sigma_{i_k}, \sigma_{j_k})$ such that $\idx_0(\sigma_{i_k})$ and $\idx_0(\sigma_{j_k})$ are consecutive integers. Let $k_* = \argmin_k \norm{\idx_0(\sigma_{i_k}) - \idx_0(\sigma_{j_k})}$. To obtain a contradiction, suppose that $s_1: = \idx_0(\sigma_{i_{k_*}})$ and $s_2:= \idx_0(\sigma_{j_{k_*}})$ are not consecutive integers. Without loss of generality, $s_1 < s_2$. For $r \in \{1, \ldots, s_2 - s_1 -1\}$, let $\tau_{s_1 + r} := \idx^{-1}_0(s_1 + r)$. That is, $\tau_{s_1+1}, \ldots, \tau_{s_2-1}$ are the simplices between $\sigma_{i_{k_*}}$ and $\sigma_{j_{k_*}}$. For all $r$, either $(\sigma_{i_{k_*}}, \tau_{s_1+r}) \in \{(\sigma_{i_k}, \sigma_{j_k})\}_{k=1}^m$ or $(\sigma_{j_{k_*}}, \tau_{s_1+r}) \in \{(\sigma_{i_k}, \sigma_{j_k})\}_{k=1}^m$. That is, either the relative order of $\sigma_{i_{k_*}}$ and $\tau_{s_1+r}$ changes or the relative order of $\sigma_{j_{k_*}}$ and $\tau_{s_1 + r}$ changes. By definition of $k_*$, none of the pairs in the set $\{(\tau_{s_1+r_1}, \tau_{s_1+r_2})\}_{r_1, r_2}$ are in $\{(\sigma_{i_k}, \sigma_{j_k})\}_{k=1}^m$. That is, the relative order of the simplices $\tau_{s_1 + 1}, \ldots, \tau_{s_2 -1}$ does not change. Therefore, either $(\sigma_{i_{k_*}}, \tau_{s_1 + 1}) \in \{(\sigma_{i_k}, \sigma_{j_k})\}_{k=1}^m$ or $(\sigma_{j_{k_*}}, \tau_{s_1 + r}) \in \{(\sigma_{i_k}, \sigma_{j_k})\}_{k=1}^m$ for all $r$. In either case, one of these is a transposition of simplices whose indices in $\idx_0$ are consecutive integers, which is a contradiction.

Now we prove the lemma by induction on $m$. When $m = 1$, we just showed that $\idx_0(\sigma_{i_1})$ and $\idx_0(\sigma_{j_1})$ are consecutive integers, so $\zeta_1 = \idx_1$. In the general case, we can assume $\idx_0(\sigma_{i_1})$ and $\idx_0(\sigma_{j_1})$ are consecutive integers without loss of generality. The set $\{(\sigma_{i_k}, \sigma_{j_k})\}_{k=2}^m$ is the set of pairs $(\sigma_{i_k}, \sigma_{j_k})$ such that
\begin{equation*}
(\zeta_1(\sigma_{i_k}) - \zeta_1(\sigma_{j_k}))\times(\idx_1(\sigma_{i_k}) - \idx_1(\sigma_{j_k})) < 0\,.
\end{equation*}
That is, $\sigma_{i_k}$ and $\sigma_{j_k}$ have different relative orders in $\zeta_1$ and $\idx_1$ for all $k \in \{2, \ldots, m\}$.
By induction, we can assume that the sequence $\{(\sigma_{i_k}, \sigma_{j_k})\}_{k=2}^m$ is ordered such that $\zeta_{k-1}(\sigma_{i_k})$ and $\zeta_{k-1}(\sigma_{j_k})$ are consecutive integers for $k \in \{2, \ldots, m\}$ and $\zeta_m = \idx_1$.
\end{proof}

\subsubsection{Modifications to step 1: Computing the polygons}\label{sec:step1mods}

In \Cref{sec:findpolygons}, we maintained a dictionary $\mathcal{D}_1(\tri)$ for each triangle $\tri \in \base$. The keys were pairs $(\sigma, \tau)$ such that $I(\sigma, \tau) \cap \tri$ was a line segment in $\tri$, and the value of a key $(\sigma, \tau)$ was the list $[v, w]$ of vertices in $\mathcal{A}(L)$ that were the endpoints of the line segment $I(\sigma, \tau) \cap \tri$.

Now we maintain two additional dictionaries $\mathcal{D}_2(\tri)$ and $\mathcal{D}_3(\tri)$ for each triangle $\tri \in \base$. We initialize these dictionaries to be empty, and we update them as we traverse the edges of $\base$. At any time in this process, the keys of $\mathcal{D}_2(\tri)$ are pairs $(v, w)$ of vertices in $\mathcal{A}(L)$ such that 
\begin{enumerate}
    \item the line segment $(v, w)$ has been detected in $\tri$,
    \item the line segment $(v, w)$ is not an edge of $\base$.
\end{enumerate}
The value of $\mathcal{D}_2(\tri)[(v, w)]$ is a list $[(\sigma_{i_1}, \sigma_{j_1}), \ldots, (\sigma_{i_m}, \sigma_{j_m})]$ of the simplex pairs that we have found so far such that $\sigma_{i_k}$ and $\sigma_{j_k}$ swap along $(v, w)$. The keys of $\mathcal{D}_3(\tri)$ are vertices $v \in \mathcal{A}(L)$ such that
\begin{enumerate}
    \item vertex $v$ has been detected along an edge $e$ of triangle $\tri$,
    \item there is a pair $(\sigma, \tau)$ of simplices such that $(\sigma, \tau)$ swaps at $v$ and we have not yet found a vertex $w$ such that $I(\sigma, \tau) \cap \tri = (v, w)$.
\end{enumerate}

We modify the algorithm of \Cref{sec:findpolygons} as follows. Suppose that we detect a vertex $v$ along edge $e$ in $\mathcal{A}(L)$ for the set $\{(\sigma_{i_1}, \sigma_{j_1}), \ldots, (\sigma_{i_m}, \sigma_{j_m})\}$ of simplex pairs. We do the following:
\begin{enumerate}
    \item {\bf Update $\mathcal{D}_1$:} For each triangle $\tri \in \base$ that is adjacent to $e$, we append $v$ to the list of vertices for $\mathcal{D}_1(\tri)[(\sigma_{i_k}, \sigma_{j_k})]$ for all $k$, as in \Cref{sec:findpolygons}.

    \vspace{3mm}
    
    \item {\bf Update $\mathcal{A}(L)$:} If $v$ is not an endpoint of $e$, we split the edge $e$ in $\mathcal{A}(L)$ and add an internal vertex in $e$, as in \Cref{sec:findpolygons}. If $v$ is an endpoint of $e$, we do not split the edge or create a new vertex because $\base$ already has a vertex at $v$.

    \vspace{3mm}
    
    \item {\bf Update $\mathcal{D}_2$, $\mathcal{D}_3$, and the edge labels: } For each triangle $\tri$ that is adjacent to $e$ and each $(\sigma_{i_k}, \sigma_{j_k})$:
    \begin{itemize}
        \item If $v$ is the only vertex in the list $\mathcal{D}_1(\tri)[(\sigma_{i_k}, \sigma_{j_k})]$, then we have not yet detected a line segment for $(\sigma_{i_k}, \sigma_{j_k})$ of the form $(v, w)$ for some vertex $w$. We do the following: If $v$ is not in $\mathcal{D}_3(\tri)$, add the key $v$ to $\mathcal{D}_3(\tri)$ with value $[(\sigma_{i_k}, \sigma_{j_k})]$. Otherwise, append $(\sigma_{i_k}, \sigma_{j_k})$ to $\mathcal{D}_3(\tri)[v]$.
        \item Otherwise, there is another vertex $w \in \mathcal{D}_1(\tri)[(\sigma_{i_k}, \sigma_{j_k})]$. This implies that we have just detected a line segment $(v, w)$ in $\tri$ for $(\sigma_{i_k}, \sigma_{j_k})$. We remove $(\sigma_{i_k}, \sigma_{j_k})$ from $\mathcal{D}_3(\tri)[w]$.
        \begin{itemize}
            \item If $v$ and $w$ are not both vertices of $\base$, then $(v, w)$ is not an edge of $\base$. We do the following: If $(v, w)$ is not in $\mathcal{D}_2(\tri)$, then add key $(v, w)$ to $\mathcal{D}_2(\tri)$ with value $[(\sigma_{i_k}, \sigma_{j_k})]$. Otherwise, append $(\sigma_{i_k}, \sigma_{j_k})$ to $\mathcal{D}_2(\tri)[(v, w)]$.
            \item Otherwise, $v$ and $w$ are both vertices of triangle $\tri$. This means that we have detected a line segment $(v, w)$ in $\tri$ in which $v$ and $w$ are the endpoints of an edge $e'$ in $\base$. If $e'$ is an edge on the boundary of $\base$, then we do nothing. Otherwise, let $\tri_2$ be the other triangle that is adjacent to $e'$. By \Cref{lem:linesegs}, the pair $(\sigma_{i_k}, \sigma_{j_k})$ swaps along $(v, w)$ if and only if the line segment is not detected in $\tri_2$. If $e'$ already stores a reference to $(\sigma_{i_k}, \sigma_{j_k})$, then we remove it because this implies that $e'$ was already detected in $\tri_2$. Otherwise, we add a reference to $(\sigma_{i_k}, \sigma_{j_k})$ on $e'$.
        \end{itemize}
    \end{itemize}
\end{enumerate}
When the traversal of the $1$-skeleton is done, we add lines to $\mathcal{A}(L)$. For every triangle $\tri \in \base$ and every key $(v, w) \in \mathcal{D}_2(\tri)$, we add a line segment with endpoints $v$, $w$ to the DCEL that represents $\mathcal{A}(L)$. For every edge in the DCEL that is a subset of the line segment $(v, w)$, we label the edge with a reference to the list $\mathcal{D}_2(\tri)[(v, w)]$, which is the list $\{(\sigma_{i_1}, \sigma_{j_1}), \ldots, (\sigma_{i_m}, \sigma_{j_m})\}$ of simplex pairs that swap along the line segment.

\subsubsection{Modifications to step 2: Computing the pairing function}\label{sec:step2mods}

We compute a path $\Gamma$ as in \Cref{sec:simplexpairs} and traverse $\Gamma$. At each step, we walk from one polygon $P_1$ to the next polygon $P_2$ by crossing an edge $e$ in $\mathcal{A}(L)$. The edge $e$ stores a list of simplex pairs $(\sigma, \tau)$ such that $\sigma$ and $\tau$ have different relative orders in the polygons $P_1$, $P_2$. We update the simplex indexing one transposition at a time. Let $\overline{\idx}: \K \to \{1, \ldots, N\}$ denote the current indexing, which we initialize to the simplex indexing $\idx_f(\cdot, P_1)$ in $P_1$. While the list that $e$ stores is nonempty, we do the following:
\begin{enumerate}
    \item Let $(\sigma, \tau)$ be the first element of the list.
    \item If $\overline{\idx}(\sigma)$ and $\overline{\idx}(\tau)$ are consecutive integers, then we update $\overline{\idx}$ by swapping the order of $\sigma$ and $\tau$. As in \Cref{sec:simplexpairs}, we also update the RU decomposition, and the (birth, death) simplex pairs. We remove $(\sigma, \tau)$ from the list.
    \item Otherwise, we move $(\sigma, \tau)$ to the end of the list.
\end{enumerate}
At the end of this algorithm, $\overline{\idx}$ is the simplex indexing $\idx_f(\cdot, P_2)$ in $P_2$ (by \Cref{lem:multiple_trans}), the RU decomposition is an RU decomposition for $P_2$, and we have computed the (birth, death) simplex pairs for $P_2$.

\end{document}